\documentclass[12pt,renqo]{amsart}

\usepackage{fullpage}
\usepackage{soul}
\usepackage{times}
\usepackage{dsfont}
\usepackage{graphicx}
\usepackage{filecontents}
\usepackage[font=small,labelfont=bf]{caption}
\usepackage{amssymb}
\usepackage{amsmath}
\usepackage{pifont}
\usepackage[english]{babel}
\usepackage[autostyle,english=american]{csquotes}
\MakeOuterQuote{"}
\usepackage{amsrefs}
\usepackage{linearb}
\usepackage{listings}
\usepackage{mathtools}
\usepackage{enumerate}
\usepackage{bm}
\usepackage{esint}
\usepackage{bbm}
\usepackage{thmtools}
\usepackage{thm-restate}
\usepackage{array}
\usepackage{subcaption}
\usepackage{hyperref}
\usepackage[dvipsnames]{xcolor}

\makeatletter
\def\@tocline#1#2#3#4#5#6#7{\relax
   \ifnum #1>\c@tocdepth 
   \else
     \par \addpenalty\@secpenalty\addvspace{#2}%
     \begingroup \hyphenpenalty\@M
     \@ifempty{#4}{%
       \@tempdima\csname r@tocindent\number#1\endcsname\relax
     }{%
       \@tempdima#4\relax
     }%
     \parindent\z@ \leftskip#3\relax \advance\leftskip\@tempdima\relax
     \rightskip\@pnumwidth plus4em \parfillskip-\@pnumwidth
     #5\leavevmode\hskip-\@tempdima #6\nobreak\relax
     \ifnum#1<0\hfill\else\dotfill\fi\hbox to\@pnumwidth{\@tocpagenum{#7}}\par
     \nobreak
     \endgroup
   \fi}
\makeatother

\let\oldtocsection=\tocsection
\let\oldtocsubsection=\tocsubsection
\let\oldtocsubsubsection=\tocsubsubsection
 
\renewcommand{\tocsection}[2]{\hspace{0em}\oldtocsection{#1}{#2}}
\renewcommand{\tocsubsection}[2]{\hspace{3em}\oldtocsubsection{#1}{#2}} 
\renewcommand{\tocsubsubsection}[2]{\hspace{5em}\oldtocsubsubsection{#1}{#2}}

\theoremstyle{plain}
\declaretheorem[numberwithin=section]{theorem}
\declaretheorem[name=Proposition,sibling=theorem]{prop}
\declaretheorem[name=Corollary,sibling=theorem]{cor}
\declaretheorem[sibling=theorem]{lemma}

\def\R{\mathbb{R}}
\def\C{\mathbb{C}}
\def\Z{\mathbb{Z}}

\def\N{\mathbb{N}}

\def\F{\mathcal{F}}

\def\E{\mathbb{E}}

\def\P{\mathbb{P}}

\def\L{{L}}

\def\1{\mathsf{1}}

\def\to{\rightarrow}

\renewcommand{\mod}[1]{\bmod{#1}}

\newcolumntype{C}{>{$}c<{$}}

\newenvironment{rezabib}
  {\bibdiv\biblist\setupbib}
  {\endbiblist\endbibdiv}
\def\setupbib{\catcode`@=\active}
\begingroup\lccode`~=`@
  \lowercase{\endgroup\def~}#1#{\gatherkey{#1}}
\def\gatherkey#1#2{\gatherkeyaux{#1}#2\gatherkeyaux}
\def\gatherkeyaux#1#2,#3\gatherkeyaux{\bib{#2}{#1}{#3}}

\title[The Distribution of Values of $\frac{L'}{L}(1/2+\epsilon,\chi_D)$]{The Distribution of Values of $\frac{L'}{L}(1/2+\epsilon,\chi_D)$}
\author{Alia Hamieh}
\author{Rory McClenagan}

 \address{Department of Mathematics and Statistics \\
        University of Northern British Columbia \\
        3333 University Way\\
        Prince George, BC V2N4Z9 \\
        Canada}
\email{alia.hamieh@unbc.ca}
\email{mcclenaga@unbc.ca}
\date{\today}


\thanks{Research of the first author is partially supported by NSERC Discovery Grant. Research of the second author was partially supported by NSERC Undergraduate Summer Research Awards}

\begin{document}
\maketitle
\begin{abstract}
    We determine the limiting distribution of the family of values $\frac{L'}{L}(1/2+\epsilon,\chi_D)$ as $D$ varies over fundamental discriminants. 
   Here, $0<\epsilon<\frac12$, and $\chi_D$ is the real character associated with $D$. Moreover, we also establish an upper bound for the rate of convergence of this family to its limiting distribution. As a consequence of this result, we derive an asymptotic bound for the small values of $\left|\frac{L'}{L}(1/2+\epsilon,\chi_D)\right|$.
\end{abstract}

\section{Introduction}\label{sec-introduction}

Many mathematicians have studied the distribution of values of $L$-functions in the critical strip. Some of the earliest results on this topic are due to Bohr-Jessen (\cite{BJ1}, \cite{BJ2}) and Jessen-Wintner \cite{JW}.  These authors obtain the distribution function of $\log\zeta(\sigma+it)$ for a fixed $\sigma>\frac12$ and established several analytic properties of this function. Another influential result on this topic is Selberg's central limit theorem \cite{Selberg} which states that values $\log\zeta(\frac12+it)$  have an approximately two-dimensional Gaussian distribution. Distribution problems for several other families of $L$-functions have been considered from various points of view  over the last 70 years. Consider for example the family of $L$-functions associated with real quadratic characters $\chi_{D}$ where $\chi_{D}(n)$ is the Kronecker symbol $\left(\frac{D}{n}\right)$. Chowla and Erdos \cite{CE} proved that the family $\{L(\sigma,\chi_{D}):D>0,D\equiv0,1\mod4\}$, for a fixed $\sigma>\frac34$, admits a continuous and strictly increasing asymptotic distribution function. Elliott also considered this particular family of $L$-values in a series of papers in the 1970's, thereby improving on the previous body of work. One of Elliott's results in this direction is the following theorem \cite[Theorem~1]{elliott}.

\begin{theorem}
Let $\sigma_0$ be a real number that satisfies $\frac12+(\log\log\log N)^{-\frac12}\leq\sigma_0\leq1$.
There exist distribution functions $F(s,z)$ so that the estimate
\[\frac{1}{\pi(N)}\#\{p\leq N: p\;\text{prime},\; |L(s,\chi_{p})|\leq e^{z}\}=F(s,z)+O((\log\log\log N)^{-2}),\quad\quad\text{as}\;N\to\infty,\] holds uniformly for all $s$ in $R_{N}=\{s:\sigma_0\leq\sigma\leq\frac54,\left|\Im(s)\right|<N^{\frac{1}{13}(2\sigma-1)}\}$ and for all real numbers $z$. For each value of $s$, the function $F(s,z)$ is infinitely differentiable with respect to $z$.
Moreover, the characteristic function $\varphi(s,\tau)$ of $F(s,z)$ has the form 
 $$\varphi(s,\tau)=\prod_{p\;\text{prime}} \frac12\left(\exp\left(-i\tau\left|\log(1+p^{-s})\right|\right)+\exp\left(-i\tau\left|\log(1-p^{-s})\right|\right) \right),$$ and satisfies the bound 
 $\varphi(s,\tau)\ll\exp\left(-\frac{c_1|\tau|^{\frac{1}{\sigma}}}{2\sigma-1}\right)$ for all $\sigma>\frac12.$ 
\end{theorem}

In an important paper \cite{GS}, Granville and Soundararajan studied the distribution of large values of  $L(1,\chi_D)$ as $D$ varies over all fundamental discriminants. One of their results implies that the proportion of fundamental discriminants $D$ with $|D|\leq x$ such that $L(1,\chi_D)\geq e^{\gamma}\tau$ decays doubly exponentially in $\tau=\log\log x$. 
In \cite{GS}, the authors  compare the distribution of the values of $ L(1,\chi_D)$ with the distribution of the probabilistic model $L(1,X)$ for some carefully chosen random variable $X$.

The idea of comparing the distribution of values of $L(1,\chi_D)$ to a random model  precedes \cite{GS}. For instance, it appears in the work of Elliott  \cite{elliott1,elliott} where he reduced the problem to a probability problem concerning sums of independent random variables on a finite probability space.

Lamzouri explored this line of research even further and established a framework for studying the distribution of large values of various families of $L$-functions inside the critical strip (see for example \cite{lamzouri1} and \cite{lamzouri2}). In \cite{lamzouri3}, Lamzouri studied the distribution of large values of $\frac{L'}{L}(1,\chi_D)$. These values have great arithmetic significance as they are directly related to the values of the Euler-Kronecker constants of the quadratic fields $\mathbb{Q}(\sqrt{D})$. In fact, the distribution of values of logarithmic derivatives of Dirichlet $L$-functions in the critical strip was initiated by Ihara and Matsumoto  (see for example \cite{Ihara1}, \cite{Ihara2}, \cite{I-M} and \cite{I-M1}).  Their approach, however, does not follow the probabilistic framework employed in  \cite{GS}, \cite{HM}, \cite{lamzouri1}, \cite{lamzouri2} and \cite{lamzouri3} among other papers. Instead, it is based on classical results such as L\'{e}vy's continuity theorem and Jessen-Wintner theory of infinite convolutions of distribution functions.

Following the method employed in \cite{I-M1}, Mourtada and Murty proved the following result (see \cite[Theorem 2]{M-M}).
\begin{theorem}
\label{MM-theorem}
Let $\sigma>\frac12$, and assume the GRH. Let $\mathcal{F}(N)$ denote the set of fundamental discriminants in the interval $[-N, N]$. Then, there exists a probability density function $M_\sigma$, such that 
    \[\lim _{N\rightarrow\infty} \frac{1}{|\mathcal{F}(N)|} \left|\left\{ D\in \mathcal{F}(N); ~ \frac{L'}{L}(\sigma,\chi_D)\leq z\right\}\right|= \int_{-\infty}^{z} M_\sigma(t) dt.\]
Moreover, the characteristic function $\varphi_{F_\sigma}(y)$ of the asymptotic distribution function $F_\sigma(z)=  \int_{-\infty}^{z} M_\sigma(t) dt$  is given by 
    \[\varphi_{F_\sigma}(y)= \prod_{p}  \left(\frac{1}{p+1}+ \frac{p}{2(p+1)}\exp{\left (  -\frac{iy\log{p}}{p^\sigma-1} \right)}+ \frac{p}{2(p+1)} \exp{\left ( \frac{iy\log{p}}{p^\sigma+1}  \right)} \right).\]
\end{theorem}

The purpose of this paper is to revisit this problem and strengthen Theorem \ref{MM-theorem} by removing the dependence on the GRH and providing an explicit error term. To this end, we follow the approach of Lamzouri \cite{lamzouri3} and employ some ideas from Lamzouri, Lester, and Radziwi{\l\l}  \cite{LLR} to compare the distribution of $\frac{L'}{L}(\sigma,\chi_D)$ to that of a probabilistic random model constructed using the independent random variables introduced in \cite{GS} (see \eqref{eqn:random-model} below).

Following the typographical convention in \cite{kowalski}, we will use sans-serif fonts, such as $\mathsf{X}$, to denote arithmetic random variables, and more standard fonts, such as $X$, for abstract random variables. Using the same letter will usually indicate that the random variable $X$ is a model of the arithmetic quantity $\mathsf{X}$.

Define $\mathcal{F}(N)$ as the set of fundamental discriminants $D$ with $|D| \leq N$, and set
\begin{align}\label{eqn:arithmetic-model}\mathsf{X}_{n,N}: 
  \mathcal{F}(N)& \to \{-1,0,1\}, \\
   D &\mapsto \chi_D(n).\nonumber
  \end{align}

Let $\{X_p\}_{p\;\text{prime}}$ be the sequence of independent random variables given by
	\begin{equation}\label{eqn:random-model}\P \big( X_p = a \big) = \begin{cases}
        \frac{p}{2(p+1)} & \text{if $a= \pm 1$}, \\
        \frac{1}{p+1} & \text{if $a=0$}.
    \end{cases}\end{equation}
Furthermore, for any positive integer $n$, define   $X_n = \prod_{p \mid n} X_p ^{\nu_p(n)}$, where $\nu_p(n)$ is the $p$-adic valuation of $n$.
The random variables $X_n$ satisfy
\begin{equation}\label{eq-E(X(n))}
    \E \big( X_n \big) = 
    \begin{cases}
        \prod_{p \mid n} \Big( \frac{p}{p+1} \Big) &\text{if $n$ is a square} \\
        0 &\text{otherwise}.
    \end{cases}
\end{equation}

The sequence $X=\{X_{n}\}_{n\in\mathbb{N}}$ was first introduced in \cite{GS} for the purpose of studying the distribution of the extreme values  of  $L(1,\chi_D)$ as $D$ varies over all fundamental discriminants.

For an odd prime $p$, consider $\chi_D(p)$ for $D \in \Z/ p^2 \Z$. Since $D$ is a fundamental discriminant, we know that the residue class corresponding to $p^2$ is not contained in $\F(N)$. For the remaining $p^2 - 1$ residue classes $\chi_D(p) = 0$ for $p-1$ of them (whenever $D$ is a multiple of $p$). The values $-1$ and $1$ on the other hand should occur equally often amongst the remaining $p^2 - p$ residue classes. 

This suggests that the random model $X$ should be a good model for the arithmetic sequence $\{\mathsf{X}_{n,N}\}_{n\in\mathbb{N}}$. In fact, one can prove that for all $k\in\N$, we have
\begin{equation}\label{eq-X_N_convergence}\lim_{N\to\infty}\E_N\big[ \mathsf{X}_{n,N}^k \big] = \E \big[ X_n^k \big],\end{equation}
where $\E_N\big[ \mathsf{X}_{n,N}^k \big]=\displaystyle{\frac{1}{\mathcal{F}(N)}\sum_{D\in\mathcal{F}(N)}X_{n,N}^k(D)}$. This follows from Lemma \ref{lemma-bridge} in Section \ref{sec-preliminary_lemmas} using the complete multiplicativity of $X_n$ and $X_{n,N}$. 
 
 Fix $\epsilon$ with $0<\epsilon<\frac12$. The objects of interest in this paper are the values $\frac{L'}{L}(1/2+\epsilon,\chi_D)$ as $D$ varies over fundamental discriminants.    Formally, we want to analyze the limiting distribution as $N \to \infty$ of the arithmetic random variables
    \begin{align*}\mathsf{L}_{\epsilon,N} :         \mathcal{F}(N) &\to \R, \\
        D& \mapsto \frac{L'}{L}\left( 1/2 + \epsilon,\chi_D\right).
    \end{align*}

For $\Re(s)>1$, we have 
   \begin{equation}\label{eq-absolute}\frac{L'}{L} \big( s, \chi_D \big) = \sum_{n=1}^\infty \frac{\Lambda(n)} {n^s} \chi_D(n)= \sum_{n=1}^\infty \frac{\Lambda(n)} {n^s} \mathsf{X}_{n,N}(D).\end{equation}

In view of \eqref{eq-X_N_convergence} and \eqref{eq-absolute},   
we introduce the abstract $\R$-valued random variable 
\begin{equation}\label{eq-L_epsilon-definition}
    \L_\epsilon = \sum_{n=1}^\infty \frac{\Lambda(n)}{n^{\frac12+\epsilon}}X_n.
\end{equation}

Using the orthogonality relation \eqref{eq-E(X(n))} and applying the Menshov-Rademacher theorem \cite[Theorem~B.10.5 ]{kowalski}, we see that the random series on the right hand side of \eqref{eq-L_epsilon-definition} is almost surely convergent, and thus $\L_\epsilon$ is well-defined. 

More generally, let $\tau>\frac12$, and let $U_{\tau}=\{s\in\mathbb{C}:\Re(s)>\tau\}$. It follows from the Menshov-Rademacher theorem that the random series \begin{equation}\label{eq-series1}\sum_{n=1}^\infty \frac{\Lambda(n)}{n^{s}}X_n\end{equation} is almost surely convergent on $U_{\tau}$, and so it defines a holomorphic function there. We also consider the random series \begin{equation}\label{eq-L_epsilon_Euler} \sum_p \frac{(\log p) X_p}{p^{s} -X_p} ,\end{equation}
which, by Kolmogorov's  theorem \cite[Theorem~B.10.1]{kowalski}, is almost surely convergent  on $U_{\tau}$, and so it defines a holomorphic function there. One could easily verify that the series \eqref{eq-series1} and \eqref{eq-L_epsilon_Euler}  are equal for all $s$ with $\Re(s)>1$. By analytic continuation, we see that
\begin{equation}\label{eqn:two-series}\sum_{n=1}^\infty \frac{\Lambda(n)}{n^{s}}X_n=\sum_p \frac{(\log p) X_p}{p^{s} -X_p}\end{equation}
almost surely in $U_{\tau}$. 
In particular, we have 
\begin{equation}\label{eq-L_epsilon-definition-1}
    \L_\epsilon = \sum_{n=1}^\infty \frac{\Lambda(n)}{n^{\frac12+\epsilon}}X_n=\sum_p \frac{(\log p) X_p}{p^{\frac12+\epsilon} -X_p}.
\end{equation}

Consider the distribution functions
    \[F_{\epsilon,N}(z) = \P_N \big( \mathsf{L}_{\epsilon,N} \leq z \big):= \frac{1}{|\mathcal{F}(N)|}\left|\left\{D\in\mathcal{F}(N): \frac{L'}{L}\left( 1/2+ \epsilon,\chi_D\right)\leq z\right\}\right|,\] and \[F_\epsilon(z) = \P \big(\L_\epsilon \leq z \big),\] for all $z \in \R$. Recall that $\mathsf{L}_{\epsilon,N}$ is said to converge in distribution to $\L_\epsilon$ if
	\begin{equation}\label{eqn:conv-dist}\lim_{N \to \infty} F_{\epsilon,N}(z) = F_\epsilon(z),\end{equation}
for every continuity point $z$ of $F_{\epsilon}$. Not only are we interested in establishing \eqref{eqn:conv-dist}, but we are also interested in determining how well the distribution  of $\L_\epsilon$ approximates that of $\mathsf{L}_{\epsilon,N}$. More precisely, the main result of this paper is the following theorem.
\begin{theorem}\label{thm-main}
Fix $0 < \epsilon < \frac12$. Then $F_{\epsilon,N}$ converges in distribution to $F_\epsilon$ which possesses a smooth density function. The characteristic function of $F_{\epsilon}$ has the form
\begin{equation}\label{eq-char_function_definition}
    \varphi_{F_{\epsilon}}(\tau) =
    \E ( \exp(i \tau  \L_\epsilon ))= \prod_{p}\Bigg(\frac{1}{p+1} + \frac{p}{2(p+1)} \bigg[ \exp\bigg(-i \tau\frac{\log p}{p^{\frac12+\epsilon}-1}\bigg) + \exp\bigg( i \tau\frac{\log p}{p^{\frac12+\epsilon}+1}\bigg) \bigg] \Bigg).
\end{equation}
Furthermore, as $N\to\infty$, we have
    \[ \| F_{\epsilon,N} - F_\epsilon \|_{\infty} \ll_\epsilon \bigg( \frac{\log \log N}{\log N} \bigg)^{\frac12+\epsilon}.\]
\end{theorem}
Using this theorem, we derive the following asymptotic bound for the small values of $\left|\frac{L'}{L}(1/2+\epsilon,\chi_D)\right|$. 
\begin{cor}
Let $\displaystyle{m_{N}=\min_{D\in\mathcal{F}(N)}\left(\left|\frac{L'}{L}(1/2+\epsilon,\chi_D)\right|\right)}$. As $N\to\infty$, we have 
\[m_{N}\ll \left(\frac{\log\log N}{\log N}\right)^{\frac12+\epsilon}.\]
\end{cor}
\begin{proof}

Let $\eta=\eta(N)$ be a positive parameter which will be chosen so that $\eta(N)\to0$ as $N\to\infty$. Let \[\Psi_N(\eta)=\left|\left\{D\in\mathcal{F}(N):\left|\frac{L'}{L}(1/2+\epsilon,\chi_{D})\right|\leq \eta\right\}\right|.\] By Theorem \ref{thm-main}, we have \[\frac{\Psi_N(\eta)}{|\mathcal{F}(N)|}= \P\left(\L_{\epsilon}\in\left[-\eta,\eta\right]\right)+O\left( \left( \frac{\log \log N}{\log N} \right)^{\frac12+\epsilon}\right).\]
Let $f_{\epsilon}(x)$ be the smooth density function associated with $F_{\epsilon}$. By \cite[Proposition~B.10.8]{kowalski} applied to the random series $\sum_p \frac{(\log p) X_p}{p^{\frac12+\epsilon} -X_p}$, we know that $f_{\epsilon}(0)>0$. It follows that

\[\P\left(\L_{\epsilon}\in\left[-\eta,\eta\right]\right)=\int_{-\eta}^{\eta}f_{\epsilon}(x)\;dx\gg \eta.\] Choosing $\eta=C\left( \frac{\log \log N}{\log N} \right)^{\frac12+\epsilon}$ for some large enough positive constant $C$ yields 
  \[\frac{\Psi_N(\eta)}{|\mathcal{F}(N)|}\gg \left( \frac{\log \log N}{\log N} \right)^{\frac12+\epsilon}. \] Hence, we get $m_{N}\ll \left(\frac{\log\log N}{\log N}\right)^{\frac12+\epsilon}$ as desired.
  \end{proof}

The corollary above is an analogue of \cite[Theorem~1.1]{LL} where the authors investigate the small values of $\left|\frac{L'}{L}(1,\chi)\right|$ for non-principal Dirichlet characters $\chi$ modulo $q$, as $q\to\infty$ over the primes.

\noindent{\bf Organization.} The structure of the paper is as follows. In Section \ref{sec-quanitative_moments}, we  prove Proposition \ref{prop-quantitative_moments} which provides a version of Berry-Esseen inequality based on the method of moments. In Section \ref{sec-method_of_proof}, we show how we use Proposition \ref{prop-quantitative_moments} to deduce Theorem \ref{thm-main} from two key results; namely, Theorem \ref{thm-moment_relation} and Proposition \ref{prop-random_moment_bound}. The former is a result relating the moments of the random model $\L_{\epsilon}$ and the arithmetic model $\mathsf{L}_{\epsilon, N}$. The latter is a decay bound on the moments of the random model $\L_{\epsilon}$. In Section \ref{sec-preliminary_lemmas}, we collect some key lemmas, allowing for streamlined proofs of these two key results. In Section \ref{sec-moment_relation}, we prove Theorem \ref{thm-moment_relation}. In Section \ref{sec-moment_decay}, we prove Proposition \ref{prop-random_moment_bound}.

\noindent{\bf Conventions and Notation.} \begin{itemize}\item Given two functions $f(x)$ and $g(x)$, we shall interchangeably use the notation  $f(x)=O(g(x))$ and $f(x) \ll g(x)$  to mean there exists $M >0$ such that $|f(x)| \le M |g(x)|$ for all sufficiently large $x$. 
We write $f(x) \asymp g(x)$ to mean that the estimates $f(x) \ll g(x)$ and $g(x) \ll f(x)$ hold simultaneously. 
\item Throughout the paper $\epsilon$ denotes a fixed positive constant with $0<\epsilon<\frac12$. 
\item The letter $p$ will always be used to denote a prime number. 
\item The capital letter $F$ is used for distribution functions and the characteristic function of a distribution function $F$ is denoted by $\varphi_F$. 
\item We denote by $\mathcal{F}(N)$ the set of all fundamental discriminants $D$ with $|D|\leq N$. 
\item For a subset $\mathcal{A}(N)$ of $\mathcal{F}(N)$, we set  $\P_{N}(\mathcal{A}(N))=\frac{|\mathcal{A}(N)|}{|\mathcal{F}(N)|}$. 
\item For an arithmetic random variable $\mathsf{Y}_{N}$ on $\mathcal{F}(N)$, we denote by $\E_{N}(\mathsf{Y}_{N})$ the average \[\frac{1}{|\mathcal{F}(N)|}\sum_{D\in\mathcal{F}(N)}\mathsf{Y}_{N}(D).\] We also use the notation $\E_{N}(\mathsf{1}_{\mathcal{A}(N)^{c}}\mathsf{Y}_{N})$ to denote the average $\frac{1}{|\mathcal{F}(N)|}\sum_{D\in\mathcal{F}(N)\setminus\mathcal{A}(N)}\mathsf{Y}_{N}(D)$.
\end{itemize}

\noindent{\bf Acknowledgements.} The authors would like to thank Amir Akbary and Edward Dobrowolski for helpful comments and discussions related to this work.

\section{Berry-Esseen Inequality}\label{sec-quanitative_moments}

There are two main tools used to prove convergence in distribution. The first tool is L\'{e}vy's continuity theorem, which relates convergence in distribution of a given sequence of distribution functions to point-wise convergence of the corresponding characteristic functions. The second tool is the method of moments which instead relies on proving the convergence of all of the integral moments of the random variables. Both of these methods are non-quantitative in their original forms. However, with some additional assumptions, we can reformulate both of these results in a quantitative format. For instance, we have the following effective analogue of L\'{e}vy's continuity theorem  (see \cite[page~431]{T}).
\begin{prop}\label{prop-Berry-Esseen}
Let $\{Y_N\}_{N=1}^\infty$ and $Y$ be real-valued random variables. Let $F_N$ and $F$ denote the corresponding distribution functions, and let $\varphi_{F_N}$ and $\varphi_{F}$ denote the corresponding characteristic functions. Suppose that $F$ is absolutely continuous with bounded density. Then we have
\begin{equation}\label{eq-Berry-Esseen}
    \|F_N - F\|_\infty \ll \frac{1}{T(N)} + \int_{-T(N)}^{T(N)} \bigg| \frac{\varphi_{F_N}(\tau) - \varphi_F(\tau)}{\tau} \bigg| d\tau,
\end{equation}
for any $T(N) \to \infty$. 
\end{prop}
This type of result, which uses effective point-wise convergence of characteristic functions to attain effective convergence in distribution, is sometimes referred to as a \textit{Berry-Esseen theorem} (although some authors reserve this term for the specific case in which the limiting distribution is normal). In \cite{LLR}, the authors utilized this approach  effectively  in combination with Beurling-Selberg functions to obtain an improved upper bound on the discrepancy between the distribution of $\zeta(s)$ on the line $\Re(s)=\frac12+\epsilon$ and that of  its random model (See \cite[Theorem ~1.1]{LLR}). 

The following proposition is a version of \eqref{eq-Berry-Esseen} based on the method of moments, and it provides the main probabilistic tool which allows us to attain the discrepancy bound in Theorem \ref{thm-main} 

\begin{prop}\label{prop-quantitative_moments}
Let $\{Y_N\}_{N=1}^\infty$ and $Y$ be real-valued random variables for which all moments exist and satisfy
    \[\E \big( |Y_N|^k \big)^{1/k}, \: \E \big( |Y|^k \big)^{1/k} = o(k), \quad \text{as $k \to \infty$}.\]
Let $F_N$ and $F$ denote the distribution functions of $Y_N$ and $Y$ respectively and suppose that $\varphi_{F}$ is absolutely integrable. Let $m(N,k)$ be some positive function such that uniformly for $k \in \Z^+$
    \[ | \E(Y_N^k) - \E(Y^k)| ^{1/k} \ll m(N,k), \quad \text{as $N \to \infty$}.\]
Suppose that there exists some function $M(N) \to \infty$ such that uniformly for $k > \log M(N)$, we have
\begin{equation}\label{eq-M(n)_construction}
    \frac{m(N,k)}{k} \ll \frac{1}{M(N)}, \quad \text{as $N \to \infty$},
\end{equation}
and uniformly for $k \leq \log M(N)$, we have
\begin{equation}\label{eq-M(n)_condition}
   M(N) \ll \log \frac{1}{m(N,k)^k}, \quad \text{as $N \to \infty$}.
\end{equation}
 Then $\{Y_N\}_{N}$ converges to  $Y$ in distribution, $F$ has a smooth density function, and 
    \[\|F_N - F\|_\infty \ll \frac{1}{M(N)}.\]
\end{prop}
\begin{proof}
 By Proposition \ref{prop-Berry-Esseen}, for any $T(N) \to \infty$ we have 
    \[\|F_N - F\|_\infty \ll \frac{1}{T(N)} + \int_{-T(N)}^{T(N)} \bigg| \frac{\varphi_{F_N}(\tau) - \varphi_F(\tau)}{\tau} \bigg| d\tau.\]

Recall that   $\varphi_{F_{N}}(\tau)= \E ( \exp(i \tau Y_N ))$ and $ \varphi_{F}(\tau)=\E ( \exp(i \tau Y ))$. 
We have
\begin{align*}
    \int_{-T(N)}^{T(N)} \bigg| \frac{\varphi_{F_N}(\tau) - \varphi_F(\tau)}{\tau} \bigg| d\tau &= \int_{-T(N)}^{T(N)} \bigg| \sum_{k=1}^\infty \frac{\E(Y_N^k) - \E(Y^k)}{k!} (i\tau)^{k-1} \bigg| d\tau \\
    &\ll \int_{-T(N)}^{T(N)} \sum_{k=1}^\infty \frac{(C_1 m(N,k))^k}{k!} \tau^{k-1} d\tau,
\end{align*}
for some absolute positive constant $C_1>0$. 
Interchanging summation and integration yields
    \[\int_{-T(N)}^{T(N)} \bigg| \frac{\varphi_{F_N}(\tau) - \varphi_F(\tau)}{\tau} \bigg| d\tau \ll \sum_{k=1}^\infty \frac{(C_1 m(N,k))^k}{k \cdot k!} T(N)^k.\]
Using Stirling's formula,  $\sqrt{2\pi}k^{k+\frac12}e^{-k}\leq k!\leq e k^{k+\frac12}e^{-k}$ for all $k\in\mathbb{N}$, we get
    \[\int_{-T(N)}^{T(N)} \bigg| \frac{\varphi_{F_N}(\tau) - \varphi_F(\tau)}{\tau} \bigg| d\tau \ll \sum_{k=1}^\infty \frac{1}{k^{3/2}} \bigg( \frac{eC_1  m(N,k)}{k} T(N) \bigg)^k.\]

  It follows that  \begin{align}\label{eqn:bound}\|F_N - F\|_\infty &\ll \frac{1}{T(N)} + \sum_{k=1}^\infty \frac{1}{k^{3/2}} \bigg( \frac{eC_1 m(N,k)}{k} T(N) \bigg)^k\nonumber\\&\ll \frac{1}{T(N)} +\sum_{k\leq \log M(N)} \frac{1}{k^{3/2}} \bigg( eC_1 \frac{m(N,k)}{k} M(N) \bigg)^k + \sum_{k > \log M(N)} \bigg( eC_1 \frac{m(N,k)}{k} M(N) \bigg)^k .\end{align}

Choose $T(N)=CM(N)$ for some positive constant $C$ to be determined later. By \eqref{eq-M(n)_construction}, we have that for sufficiently large $n$,
    \[\sup_{k > \log M(N)} \frac{m(N,k)}{k} \leq \frac{C_2}{M(N)},\] for some absolute positive constant $C_2$.
Hence,
    \begin{equation}\label{eqn:2}\sum_{k > \log M(N)} \bigg( eC_1 \frac{m(N,k)}{k} T(N) \bigg)^k \ll \sum_{k > \log M(N)} \bigg( {eCC_1C_2} \bigg)^k \ll \frac{1}{M(N)},\end{equation} provided that $C<\frac{1}{e^2C_1C_2}$.
Notice that \eqref{eq-M(n)_condition} implies that there exists $C_3>0$ such that
    \[\sup_{k \leq \log M(N)} m(N,k)^k \ll e^{-C_3M(N)}, \quad \text{as $N \to \infty$}.\]
This and another application of Stirling's formula implies
    \begin{equation}\label{eqn:3}\sum_{k\leq \log M(N)} \frac{1}{k^{3/2}} \bigg( eC_1 \frac{m(N,k)}{k} T(N) \bigg)^k \ll e^{-C_3 M(N)} \sum_{k=0}^\infty \frac{(CC_1 M(N))^k}{k!} \ll e^{(CC_1-C_3)M(N)} \ll \frac{1}{M(N)},\end{equation}
provided that $CC_1<C_3$. Choosing $0<C<\min(\frac{1}{e^2C_1C_2},\frac{C_3}{C_1})$ and combining \eqref{eqn:bound}, \eqref{eqn:2} and \eqref{eqn:3} yield the desired result.
\end{proof}

\section{Proof of Theorem \ref{thm-main}}\label{sec-method_of_proof}

The proof of Theorem \ref{thm-main} is accomplished in two parts. The first part consists of proving that the large moments $\E_N \big( \1_{\mathcal{E}(N)^c} \mathsf{L}_{\epsilon,N}^k \big)$ of $\mathsf{L}_{\epsilon,N}$ defined as the average of $\frac{L'}{L}(1/2+\epsilon,\chi_D)^k$ over $D\in\mathcal{F}(N)\setminus\mathcal{E}(N)$ can be approximated by the corresponding moments $\E \big( \L_\epsilon^k \big)$ of the random model $\L_{\epsilon}$. Here $\mathcal{E}(N)$ is an exceptional set of fundamental discriminants  such that $\left|\mathcal{E}(N) \right| =O\left( N^{1-c}\right)$ for some $c>0$. More precisely, we prove the following theorem.

\begin{theorem}\label{thm-moment_relation}
There exists a set of fundamental discriminants $\mathcal{E}(N)\subset \mathcal{F}(N)$ with $\P_N \big( \mathcal{E}(N) \big) =O\left( N^{-c}\right)$ for some $c>0$, such that uniformly for $k \in \Z^+$, we have
	\[\big| \E_N \big( \1_{\mathcal{E}(N)^c} \mathsf{L}_{\epsilon,N}^k \big) - \E \big( \L_\epsilon^k \big) \big|^{1/k} \ll \frac{\log N}{N^{\frac{\epsilon^2 (\epsilon + 3) }{  12k} }}.\]
Furthermore, this holds when $\mathcal{E}(N)$ is replaced by any $\mathcal{E}_{\star}(N) \supset \mathcal{E}(N)$ as long as $\P_N(\mathcal{E}_{\star}(N)) \asymp \P_N(\mathcal{E}(N))$.
\end{theorem}

The second part of the proof of Theorem \ref{thm-main} consists of using the Berry-Essen inequality described in Proposition \ref{prop-quantitative_moments} to relate the distribution functions $F_{\epsilon,N}$ and $F_{\epsilon}$ to the moments $ \E_N \big( \1_{\mathcal{E}(N)^c} \mathsf{L}_{\epsilon,N}^k \big)$ and $\E \big( \L_\epsilon^k \big)$. This allows us to get an upper bound on the rate of convergence of $F_{\epsilon,N}$ to $F_{\epsilon}$. 

We require the following two propositions in order to verify that the conditions of Proposition 
\ref{prop-quantitative_moments} are satisfied.

\begin{prop}\label{prop-random_moment_bound}
As $k\to\infty$, we have $\E\big( |\L_\epsilon|^k \big)^{1/k} \ll k^{\frac12 -\epsilon}$.
\end{prop}

\begin{prop}[ Lemma~4 of \cite{M-M}]\label{prop-char_function_decay}
As $|\tau|\to\infty$, we have $\varphi_{F_\epsilon}(\tau) \ll \exp \big(-C |\tau|^{\frac{1}{\frac12+\epsilon}} \big)$,
for some positive constant $C$ that depends only on $\epsilon$.
\end{prop}
A proof of Proposition \ref{prop-char_function_decay} can be found in \cite{M-M}. The reader is referred to Section \ref{sec-moment_relation} and Section \ref{sec-moment_decay} for the proofs of Theorem \ref{thm-moment_relation} and Proposition \ref{prop-random_moment_bound} respectively. 

Finally, we need the following result which is inspired by \cite[Lemma 3.4]{LLR} and follows from Theorem \ref{thm-moment_relation} and Proposition \ref{prop-random_moment_bound}.

\begin{lemma}\label{lem-tail_bound}
There exists a constant $B=B(\epsilon)>0$ such that
	\[\P_N \bigg[ |\1_{\mathcal{E}(N)^c} \mathsf{L}_{\epsilon,N}| \geq \bigg( \frac{\log N}{\log \log N} \bigg)^{\frac12-\epsilon} \bigg] \ll \exp \bigg( - B \frac{\log N}{\log \log N} \bigg).\]
\end{lemma}

\begin{proof}
Markov's inequality (see, for example, \cite[Eq.~ 5.31]{Bill}) gives
\begin{equation}\label{eq-markov}
    \P_N \bigg[ |\1_{\mathcal{E}(N)^c} \mathsf{L}_{\epsilon,N}| \geq \bigg( \frac{\log N}{\log \log N} \bigg)^{\frac12-\epsilon} \bigg] \leq \bigg( \E_N \big( |\1_{\mathcal{E}(N)^c} \mathsf{L}_{\epsilon,N}|^k \big)^{1/k}  \bigg( \frac{\log \log N}{\log N} \bigg)^{\frac12-\epsilon} \bigg)^k,
\end{equation}
for any positive integer $k$. By Theorem \ref{thm-moment_relation} and Proposition \ref{prop-random_moment_bound}, we have
    \begin{equation}\label{eqn:thm3.1Prop3.2} \E_N \big( \1_{\mathcal{E}(N)^c} \mathsf{L}_{\epsilon,N}^k \big) ^{1/k} \leq C \frac{\log N}{N^{\frac{\epsilon^2 (\epsilon + 3) }{  12k}}} + k^{\frac12-\epsilon},\end{equation}for some positive constant $C$.
Choosing $k = \frac{\log N}{A \log \log N}$ with $A>\max(C^{\frac{1}{\frac12-\epsilon}},\frac{\frac12+\epsilon}{\frac{\epsilon^2 (\epsilon + 3) }{  12}})$ and combining \eqref{eq-markov} and \eqref{eqn:thm3.1Prop3.2} yield the desired result.
\end{proof}

\begin{proof}[Proof of Theorem \ref{thm-main}]
Let $\mathcal{E}_{\star}(N)=\mathcal{E}(N) \cup \{D\in\mathcal{F}(N) : |\frac{L'}{L}(1/2+\epsilon,\chi_{D})| \geq \big( \frac{\log N}{\log \log N} \big)^{\frac12-\epsilon}\}$. 
 By Lemma \ref{lem-tail_bound}, we have $\P_N(\mathcal{E}_{\star}(N)) \asymp \P_N(\mathcal{E}(N))$.  Thus, we may apply Theorem \ref{thm-moment_relation} to obtain
	\begin{equation}\label{eqn:E'bound1}\big| \E_N \big(\left( \1_{\mathcal{E}_{\star}(N)^c} \mathsf{L}_{\epsilon,N}\right)^k \big) - \E \big( \L_\epsilon^k \big) \big|^{1/k} 
	\ll \frac{\log N}{N^{\delta/k }},\end{equation}
where $\delta = \epsilon^2(\epsilon + 3) / 12 >0$. On the other hand, by our definition of $\mathcal{E}_{\star}(N)$ and Proposition \ref{prop-random_moment_bound} we have
 \begin{equation}\label{eqn:E'bound2}\big| \E_N \big(\left( \1_{\mathcal{E}_{\star}(N)^c} \mathsf{L}_{\epsilon,N}\right)^k \big) - \E \big( \L_\epsilon^k \big) \big|^{1/k} \ll k^{\frac12-\epsilon} + \big( \frac{\log N}{\log \log N} \big)^{\frac12-\epsilon}. \end{equation} Since the first term on the right hand side of \eqref{eqn:E'bound2} is dominant as long as $k  \gg \frac{\log N}{\log \log N} $, we combine \eqref{eqn:E'bound1} and  \eqref{eqn:E'bound2} to get
	\[\big| \E_N \big( \left(\1_{\mathcal{E}_{\star}(N)^c} \mathsf{L}_{\epsilon,N}\right)^k \big) - \E \big( \L_\epsilon^k \big) \big|^{1/k} 
	\ll  m(N,k),\]	
where
	\[m(N,k) = \begin{cases}
		\frac{\log N}{N^{\delta/k}} & \text{if $k \leq \frac{\delta}{\frac12+\epsilon} \frac{\log N}{\log \log N}$,} \vspace{5pt} \\ 
		k^{\frac12-\epsilon} & \text{if $k > \frac{\delta}{\frac12+\epsilon} \frac{\log N}{\log \log N}$.}
	\end{cases}\]
Observe that
	\[\sup_{k \gg \log \log N} \frac{m(N,k)}{k} \ll \bigg( \frac{\log \log N}{ \log N} \bigg)^{\frac12 + \epsilon},\]
and
	\[\inf_{k \ll \log \log N} k \log \frac{1}{m(N,k)} \gg \inf_{k \ll \log \log N} k \log \frac{N^{\delta/k}}{\log N} \gg \log N. 	\]
It follows that conditions (\ref{eq-M(n)_construction}) and (\ref{eq-M(n)_condition}) are satisfied. Since the characteristic function of the random model is absolutely integrable by Proposition \ref{prop-char_function_decay}, we can apply Proposition \ref{prop-quantitative_moments} to get
	\[\| F_{\epsilon,N;\star} - F_{\epsilon} \|_\infty \ll \bigg( \frac{\log \log N}{ \log N} \bigg)^{\frac12 + \epsilon},\] where $F_{\epsilon,N;\star}$ is the distribution function corresponding to  $\1_{\mathcal{E_{\star}}(N)^c} \mathsf{L}_{\epsilon,N}$. 
Finally, combining this with Lemma \ref{lem-tail_bound} implies
	\[\|F_{\epsilon,N} - F_{\epsilon}\|_\infty \ll \exp \bigg( - B \frac{\log N}{\log \log N} \bigg) + \| F_{\epsilon,N;\star} - F_{\epsilon} \|_\infty \ll \bigg( \frac{\log \log N}{ \log N} \bigg)^{\frac12 + \epsilon},\]
as desired.
\end{proof}

\section{Preliminary Lemmas}\label{sec-preliminary_lemmas}
Recall that if $\Re(s) >1$, we have 
    \[\bigg( \frac{L'}{L} \big( s, \chi_D \big) \bigg)^k = (-1)^k \sum_{n=1}^\infty \frac{\Lambda_k(n)}{n^s}\chi_D(n),\]
where
    \[\Lambda_k(n) = \sum_{\substack{n_1,n_2,\dots,n_k \geq 1 \\ n_1n_2 \dots n_k = n}} \Lambda(n_1)\Lambda(n_2) \dots \Lambda(n_k),\]
and it satisfies
\begin{equation}\label{eq-Lambda_k_bound}
    \Lambda_k(n) \leq \Bigg( \sum_{m \mid n} \Lambda(m) \Bigg)^k = (\log n)^k.
\end{equation}
For $0<\sigma<1$, this upper bound along with an application of partial summation yields
\begin{equation}\label{eq-Lambda_partial_sums}
    \sum_{n \leq \lambda} \frac{\Lambda_k(n)}{n^\sigma} \ll (\log \lambda)^k \lambda^{1-\sigma}.
\end{equation}

In what follows, we collect several basic lemmata that are required in the sequel.
\begin{lemma}[ Corollary 5.3 of \cite{MV}]\label{lemma-Perron}
Consider the Dirichlet series $\alpha(s) = \sum_{n=1}^\infty \frac{a_n}{n^s}$
with abscissa of absolute convergent $\sigma_a$ and abscissa of convergence $\sigma_c$. Fix some $\sigma_0 >\sigma_c$. Choose $c$ such that $c > \max (0, \sigma_a -\sigma_0)$.
Let $\lambda>0$ be non-integral. Then,
    \[\sum_{n \leq \lambda} \frac{a_n}{n^{\sigma_0}} = \frac{1}{2\pi i} \int_{c-iT}^{c+iT} \alpha(\sigma_0 + s) \frac{\lambda^s}{s} \, ds + E\]
where
    \[E \ll \sum_{\lambda/2 < n < 2 \lambda} \frac{|a_n|}{n^{\sigma_0}} \min \bigg( 1, \frac{\lambda}{T|\lambda-n|} \bigg) + \frac{4^c + \lambda^c}{T} \sum_{n=1}^\infty \frac{|a_n|}{n^{\sigma_0 + c}}.\]
\end{lemma}

\begin{lemma}[Lemma 2.2 of \cite{lamzouri3}]
\label{lemma-bound_L'/L}
Suppose that $L(s,\chi_D)$ is nonzero for $\Re(s)>\sigma_0$ and $|\Im(s)| \leq |t|+1$. Then, for $\sigma > \sigma_0$, we have
    \[\frac{L'}{L}(\sigma + it, \chi_D) \ll \frac{\log (|D|(|t|+2))}{\sigma-\sigma_0}.\]
\end{lemma}
\begin{lemma}[Theorem 3 of \cite{heath-brown}]\label{result-zero_density}
Let $N(\upsilon, \tau,\chi_D)$ denote the number of zeros of $L(s,\chi_D)$ in the rectangle
   $R(\upsilon, \tau) = \{s \in \C : \frac12 +\upsilon < \Re(s) \leq 1, \: |\Im(s)| \leq \tau\}$.
Then for any $\delta >0$, we have
    \[\sum_{D \in \F(N)} N(\upsilon, \tau,\chi_D) \ll_\delta (N\tau)^\delta N^{\frac{3-6\upsilon}{3-2\upsilon}}\tau^{\frac{4-4\upsilon}{3-2\upsilon}}.\]
\end{lemma}
In view of this lemma,  if we let $\mathcal{E}(\upsilon, \tau; N)$ denote the set of $D \in \F(N)$ for which $\frac{L'}{L}(s,\chi_D)$ has at least one pole in $R(\upsilon,\tau)$, then
    \begin{equation}\label{eqn-zero_density}|\mathcal{E}(\upsilon, \tau; N)| \ll_\delta (N\tau)^\delta N^{\frac{3-6\upsilon}{3-2\upsilon}}\tau^{\frac{4-4\upsilon}{3-2\upsilon}}.\end{equation}
Finally, the following lemma serves a crucial role as a bridge from the arithmetic random setting into the abstract probabilistic setting.
\begin{lemma}\label{lemma-bridge}
For sufficiently large $N$, we have $\E_N\big(\mathsf{X}_{n,N} \big) - \E \big(X_n \big) \ll N^{-\frac12}n^{\frac14} \log n $.
\end{lemma}
\begin{proof} 
By definition, we have
    \begin{equation}\label{eqn:char-sum}\E_N\big(\mathsf{X}_{n,N} \big) = \frac{1}{|\mathcal{F}(N)|} \sum_{D \in \mathcal{F}(N)} \chi_D(n),\end{equation}
where
\[|\mathcal{F}(N)|=\frac{6}{\pi^2}N+O\left(N^{\frac12+\epsilon}\right).\]
In fact, if $n=m^2$, then we have the following standard estimate (see for example \cite[page~640]{lamzouri3})
    \[\sum_{D \in \mathcal{F}(N)} \chi_D(m^2) =  \sum_{\substack{D \in \mathcal{F}(N) \\ (D,m) = 1}} 1 = \frac{6}{\pi^2} N \prod_{p | m} \bigg( \frac{p}{p+1} \bigg) + O \big( N^{\frac12} \tau(m) \big),\]
where $\tau(m)$ is the divisor function. Combining this with (\ref{eq-E(X(n))}) yields \[\E_N\big(\mathsf{X}_{n,N} \big) - \E \big(X_n \big) \ll N^{-\frac12}\tau(\sqrt{n}),\]
provided that $n$ is a perfect square. By \cite[lemma~4.1]{GS}, we have 
    \[\sum_{D \in \mathcal{F}(N)} \chi_D(n) \ll N^{\frac12} n^{\frac14} \log n,\]
for non-square $n$. This implies that
   $\E_N\big(\mathsf{X}_{n,N} \big) - \E \big(X_n\big) \ll N^{-\frac12}n^{\frac14} \log n$ 
if $n$ is not a perfect square. \end{proof}

\section{Proof of Theorem \ref{thm-moment_relation}}\label{sec-moment_relation}
The point of departure in proving Theorem \ref{thm-moment_relation} is approximating integral powers of $\frac{L'}{L}(1/2+\epsilon,\chi_D)$ by short Dirichlet polynomials. 

Let $d$ and $e$ be two positive constants such that $d<e<\epsilon$. We set \[\mathcal{E}_{d,e}(N) = \{ D \in \F(N) : L(s,\chi_D)=0\; \text{for some}\;s\in R_{d,e}\},\]
where
    \[R_{d,e}= \{ s \in \C : \frac12 + (\epsilon - e) < \Re(s) \leq 1, \: |\Im (s)| \leq \lambda^{\frac12-(\epsilon - d)} + 1\}.\]
For simplicity, we suppress the subscripts from our notation and set $\mathcal{E}(N)=\mathcal{E}_{d,e}(N)$ and $R=R_{d,e}$. It follows from \eqref{eqn-zero_density} that for any $\delta>0$, we have
 \begin{equation}\label{eq-size_bad_D}\P_N \big( \mathcal{E}(N) \big) \ll \frac{1}{N^b},\end{equation}
with
    \[b=\frac{4}{3-2(\epsilon - e)}\bigg[(\epsilon - e) -  (1-(\epsilon - e)) (\frac12-(\epsilon - d)) \bigg( \frac{\log \lambda}{\log N} \bigg)\bigg] - \delta.\]

\begin{prop}\label{prop-short_dirichlet_polynomial}
Suppose $\lambda$ satisfies $\log \lambda \ll \log N$. Then, for all $D \in \mathcal{F}(N) \setminus \mathcal{E}(N)$ and $k \in \Z^+$, we have
    \[\bigg(\frac{L'}{L}\Big( \frac{1}{2}+\epsilon, \chi_D \Big)\bigg)^k = (-1)^k \sum_{n \leq \lambda} \frac{\Lambda_k(n)}{n^{\frac{1}{2}+\epsilon}} \chi_D(n) + O \bigg(\frac{(C \log N)^{k+1}}{\lambda^{d}} \bigg),\] for some positive constant $C$.
\end{prop}
\begin{proof}
Assume throughout that $D \in \mathcal{F}(N) \setminus \mathcal{E}(N)$. Lemma \ref{lemma-Perron}  gives
\begin{equation}\label{eqn-perron}
    \frac{1}{2\pi i} \int_{c-iT}^{c+iT} \bigg(-\frac{L'}{L}\Big(\frac{1}{2}+\epsilon+s,\chi_D\Big) \bigg)^k \frac{\lambda^s}{s} ds = \sum_{n \leq \lambda} \frac{\Lambda_k(n) \chi_D(n)}{n^{\frac12+\epsilon}} + \Psi_1 + \Psi_2,
\end{equation}
where
    \[\Psi_1 \ll \sum_{\lambda/2 < n < 2 \lambda} \frac{\Lambda_k(n)}{n^{\frac12+\epsilon}} \min \bigg( 1, \frac{\lambda}{T|\lambda-n|} \bigg) \quad \text{and} \quad \Psi_2 \ll  \frac{4^c + \lambda^c}{T} \sum_{n=1}^\infty \frac{\Lambda_k(n)}{n^{\frac12+\epsilon+c}}.\]
We fix $c=\frac12-\epsilon + 1/\log \lambda$ and assume without loss of generality that $\lambda \in \Z + \frac12$. Using \eqref{eq-Lambda_k_bound} we get
    \[\Psi_1 \ll \frac{\lambda^{\frac12-\epsilon}(\log (2\lambda))^k}{T} \sum_{n \leq 2\lambda}  \frac{1}{|n-\lambda|} \ll \frac{\lambda^{\frac12-\epsilon}(\log(2\lambda))^{k+1}}{T}\]
and
    \[\Psi_2 \ll \frac{\lambda^{\frac12-\epsilon}}{T} \bigg( \sum_{n=1}^\infty \frac{\Lambda(n)}{n^{1+1/\log \lambda}} \bigg)^k \ll \frac{\lambda^{\frac{1}{2}-\epsilon}}{T}(\log\lambda)^k.\]
We now shift the line of integration in \eqref{eqn-perron} from $\Re(s)=c$ to  $\Re(s) = -d$. Since $D \not \in \mathcal{E}(N)$, the integrand has only a simple pole at $s=0$. By the residue theorem, we get
    \[\bigg(-\frac{L'}{L}\Big( \frac{1}{2}+\epsilon, \chi_D \Big)\bigg)^k = \sum_{n \leq \lambda} \frac{\Lambda_k(n)}{n^{\frac{1}{2}+\epsilon}} \chi_D(n) + \Psi_1 + \Psi_2 + \Psi_3 + \Psi_4,\]
where
    \[\Psi_3 + \Psi_4 = \bigg( \int_{c+iT}^{-d+iT} + \int_{-d + iT}^{-d-iT} + \int_{-d-iT}^{c-iT} \bigg) \bigg(-\frac{L'}{L}\Big( \frac{1}{2}+\epsilon + s, \chi_D \Big)\bigg)^k \frac{\lambda^s}{s} \, ds.\]
Here, $\Psi_3$ denotes the first and third integral, and $\Psi_4$ denotes the second integral.  Applying Lemma \ref{lemma-bound_L'/L} gives
    \[\frac{L'}{L}\Big( \frac{1}{2}+\epsilon + \sigma + it, \chi_D \Big) \ll \frac{\log (|D|(|t|+2))}{\sigma+e} \leq \frac{\log (N(|t|+2))}{\sigma+e},\]
for $|t| \leq T$. It follows that
    \[\Psi_3 \ll \int_{-d}^c \bigg|\frac{L'}{L}\Big( \frac{1}{2}+\epsilon+\sigma +iT, \chi_D \Big) \bigg|^k \frac{\lambda^\sigma}{T} \, d\sigma
    \ll \frac{\lambda^c}{T \log \lambda} \bigg( \frac{\log (N(T+2))}{e-d} \bigg)^k.\]
Similarly,
\begin{align*}
    \Psi_4 &\ll \int_{-T}^T \bigg|\frac{L'}{L}\Big( \frac{1}{2} + \epsilon - d +it, \chi_D \Big) \bigg|^k \bigg| \frac{\lambda^{-d+it}}{-d+it} \bigg| \, dt 
    \ll  \frac{1}{\lambda^d} \bigg( \frac{\log(N(T+2))}{e-d} \bigg)^k \log (1+T/d).
\end{align*}
If we set $E=\Psi_1 + \Psi_2+\Psi_3 + \Psi_4$, then
\begin{align*}
    E &\ll \frac{\lambda^{\frac12-\epsilon}(\log(2\lambda))^{k+1}}{T} + \big(\log (N(T+2)) \big)^{k+1} \bigg[ \frac{\lambda^{\frac12-\epsilon}}{T \log \lambda} + \frac{1}{\lambda^d} \bigg].
\end{align*}
Choosing $T= \lambda^{\frac12-(\epsilon - d)}$ and assuming $\log\lambda\ll\log N$ yields the desired result.
\end{proof}


\begin{proof}[Proof of Theorem \ref{thm-moment_relation}]
Choose $\lambda$ such that $\log \lambda \ll \log N$. Using Proposition \ref{prop-short_dirichlet_polynomial}, \eqref{eq-Lambda_partial_sums} and \eqref{eq-size_bad_D} gives
    \[\E_N \big( \1_{\mathcal{E}(N)^c} \mathsf{L}_{\epsilon,N}^k \big) = \sum_{n \leq \lambda} \frac{\Lambda_k(n)}{n^{\frac12+\epsilon}} \E_N \big(\mathsf{X}_{n,N}\big) + O \bigg[ \frac{(C\log N)^{k+1}}{\lambda^{d}} + \frac{\lambda^{\frac12-\epsilon} (\log \lambda)^k}{N^b} \bigg].\]
Note that this equation still holds if $\mathcal{E}(N)$ is replaced by some larger exceptional set $\mathcal{E}_{\star}(N)$ as long as $\P_N(\mathcal{E}_{\star}(N)) \asymp \P_N(\mathcal{E}(N))$. We apply Lemma \ref{lemma-bridge} to obtain
    \[\E_N \big( \1_{\mathcal{E}(N)^c} \mathsf{L}_{\epsilon,N}^k \big) = \E \bigg( \sum_{n \leq \lambda} \frac{\Lambda_k(n)}{n^{\frac12+\epsilon}} X_n \bigg) + O \bigg[ (C\log N)^{k+1} \bigg( \frac{1}{\lambda^{d}} + \frac{\lambda^{\frac12-\epsilon}}{N^b} +  \frac{ \lambda^{3/4-\epsilon}}{N^{1/4}} \bigg) \bigg].\]
The orthogonality property of $X_n$ (see \eqref{eq-E(X(n))}) implies
    \[\E \bigg( \sum_{n > \lambda} \frac{\Lambda_k(n)}{n^{\frac12+\epsilon}} X_n \bigg) \ll \frac{(2\log \lambda)^k} {\lambda^{1+2\epsilon}},\]
which can clearly be neglected. Thus,
    \[\E_N \big( \1_{\mathcal{E}(N)^c} \mathsf{L}_{\epsilon,N}^k \big) =  \E \big( \L_\epsilon^k \big) + O \bigg( \frac{(C \log N)^{k+1}}{N^t} \bigg),\]
where
    \[t = \min\bigg[d \frac{\log \lambda}{\log N}, 
    \min \bigg( b, \frac{1}{4}\Big( 1- \frac{\log \lambda}{\log N} \Big) \bigg)- \frac{\log \lambda}{\log N} (\frac12-\epsilon) \bigg].\]
Choosing $\delta$ sufficiently small, $d=\epsilon - (3- \sqrt{9-6\epsilon})/2$, $e$ sufficiently close to $d$, and
    \[\frac{\log \lambda}{\log N} = \min \bigg( \frac{6}{7} \frac{4 (\epsilon- d)}{(3-2(\epsilon- d))(1-2(\epsilon- d))},\frac{1}{3- 4(\epsilon- d)}\bigg),\]
gives $t > \epsilon^2 (\epsilon + 3 ) / 12$.
\end{proof}
\section{Proof of Proposition \ref{prop-random_moment_bound}}\label{sec-moment_decay}
\begin{proof}[Proof of Proposition \ref{prop-random_moment_bound}] 
Using \eqref{eq-L_epsilon-definition-1} and applying Minkowski's inequality yield
    \begin{equation}\label{eqn:E}\E \big( |\L_\epsilon| ^k \big)^{1/k} 
    \leq \E \Bigg[ \bigg| \sum_p \frac{\log p }{p^{\frac12+\epsilon}} X_p \bigg|^k \Bigg]^{1/k} 
    + \E \Bigg[ \bigg| \sum_{p^j, \: j \geq 2} \log p \Big( \frac{X_p}{p^{\frac12+\epsilon}} \Big)^j \bigg|^k \Bigg]^{1/k}.\end{equation}
The second sum on the right hand side of \eqref{eqn:E} is
    \[ \E \Bigg[ \bigg| \sum_p \log p \frac{X_p^2}{p^{1+2\epsilon} - p^{\frac12+\epsilon} X_p} \bigg|^k \Bigg]^{1/k} 
    \ll \sum_p \frac{\log p}{p^{1+2\epsilon}} \ll 1.\]
 We split the first sum on the right hand side of \eqref{eqn:E} at some $y$ which we determine later to get
\begin{align*}
   \E \big( |\L_\epsilon| ^k \big)^{1/k} 
    &\ll \E \Bigg[ \bigg| \sum_{p \leq y} \frac{\log p }{p^{\frac12+\epsilon}} X_p \bigg|^k \Bigg]^{1/k} + \E \Bigg[ \bigg| \sum_{p > y} \frac{\log p }{p^{\frac12+\epsilon}} X_p \bigg|^k \Bigg]^{1/k}+1 \\
    &\ll \sum_{p \leq y} \frac{\log p}{p^{\frac12+\epsilon}} + \E \Bigg[ \bigg( \sum_{p > y} \frac{\log p }{p^{\frac12+\epsilon}} X_p \bigg)^{2k} \Bigg]^{1/2k}+1,
\end{align*}
by Minkowski's inequality and the Cauchy-Schwartz inequality. 
Partial summation and the prime number theorem give $\sum_{p \leq y} \frac{\log p}{p^{\frac12+\epsilon}} \ll y^{\frac12-\epsilon}$.
Observe that
    \begin{align}\label{eqn:E1} \E \Bigg[ \bigg( \sum_{p > y} \frac{\log p }{p^{\frac12+\epsilon}} X_p \bigg)^{2k} \Bigg]^{1/2k} &= \E \Bigg[ \sum_{\substack{p_1, \dots, p_k > y \\ q_1, \dots, q_k > y}} \frac{ \log p_1 \dots \log p_k \log q_1 \dots \log q_k}{(p_1\dots p_k q_1 \dots q_k)^{\frac12+\epsilon}}  X_{p_1 \dots p_k } X_{q_1 \dots q_k} \Bigg]^{1/2k}\nonumber\\&
    \leq \Bigg[\sum_{p_1, \dots p_k > y} \frac{ ( \log p_1 \dots \log p_k )^2}{(p_1\dots p_k )^{1+2\epsilon}} \sum_{\substack{q_1,\dots,q_k > y \\ q_1 \dots q_k = p_1 \dots p_k}} 1 \Bigg]^{1/2k},\end{align}
where the last inequality follows from the orthogonality of $X_n$  (see \eqref{eq-E(X(n))}). The innermost sum in \eqref{eqn:E1} counts the number of permutations on the set $\{1,\dots,k\}$, which is just $k!$. An application of Stirling's formula then yields
    \[\E \Bigg[ \bigg( \sum_{p > y} \frac{\log p }{p^{\frac12+\epsilon}} X_p \bigg)^{2k} \Bigg]^{1/2k} 
    \ll \sqrt{k} \bigg( \sum_{p > y} \frac{\log p}{p^{1+2\epsilon}} \bigg)^{\frac12}
    \ll \frac{\sqrt{k}}{y^\epsilon}.\]
Choosing $y= k$ gives the desired result.
\end{proof}

\begin{rezabib}

\bib{Bill}{book} {
author={Billingsley, Patrick},
   title={Probability and measure},
   series={Wiley Series in Probability and Statistics},
   edition={Anniversary Edition},
   publisher={John Wiley \& Sons, Inc., New York},
   date={2012},
   pages={xvii+624},
}

\bib{BJ1}{article}{
   author={Bohr, Harald},
   author={Jessen, B\"{o}rge},
   title={\"{U}ber die Werteverteilung der Riemannschen Zetafunktion, erste Mitteilung},
   language={German},
   journal={Acta Math.},
   volume={54},
   date={1930},
   number={1},
   pages={1--35},
}
\bib{BJ2}{article}{
   author={Bohr, Harald},
   author={Jessen, B\"{o}rge},
   title={\"{U}ber die Werteverteilung der Riemannschen Zetafunktion},
   language={German},
   journal={Acta Math.},
   volume={58},
   date={1932},
   number={1},
   pages={1--55},
}

\bib{CE}{article} {
   author={Chowla, Sarvadaman},
   author={Erd\"{o}s, Paul},
   title={A theorem on the distribution of the values of $L$-functions},
   journal={J. Indian Math. Soc. (N.S.)},
   volume={15},
   date={1951},
   pages={11--18},
 
}

\bib{elliott1}{article}{
   author={Elliott, P. D. T. A.},
   title={On the distribution of the values of Dirichlet $L$-series in the
   half-plane $\sigma >{1\over 2}$},
   journal={Nederl. Akad. Wetensch. Proc. Ser. A {\bf 74}=Indag. Math.},
   volume={33},
   date={1971},
   pages={222--234},
  
}
	
\bib{elliott-02}{article}{
   author={Elliott, P. D. T. A.},
   title={On the distribution of ${\rm arg}L(s,\,\chi )$ in the half-plane
   $\sigma >{1\over 2}$},
   journal={Acta Arith.},
   volume={20},
   date={1972},
   pages={155--169},
  
}

\bib{elliott}{article}{
   author={Elliott, P. D. T. A.},
   title={On the distribution of the values of quadratic $L$-series in the
   half-plane $\sigma >{1\over 2}$},
   journal={Invent. Math.},
   volume={21},
   date={1973},
   pages={319--338},
 
}

\bib{GS}{article}{
   author={Granville, Andrew},
   author={Soundararajan, Kannan},
   title={The distribution of values of $L(1,\chi_d)$},
   journal={Geom. Funct. Anal.},
   volume={13},
   date={2003},
   number={5},
   pages={992--1028},

}

\bib{heath-brown}{article}{
AUTHOR = { Heath-Brown, Roger D.},
     TITLE = {A mean value estimate for real character sums},
   JOURNAL = {Acta. Arith.},
    VOLUME = {72},
    NUMBER = {3},
      YEAR = {1995},
     PAGES = {235--275},

}

\bib{HM}{article} {
     author={Hattori, Tetsuya},
   author={Matsumoto, Kohji},
   title={A limit theorem for Bohr-Jessen's probability measures of the
   Riemann zeta-function},
   journal={J. Reine Angew. Math.},
   volume={507},
   date={1999},
   pages={219--232},
 
}

\bib{Ihara1}{article}{
   author={Ihara, Yasutaka},
   title={On the Euler-Kronecker constants of global fields and primes with
   small norms},
   conference={
      title={Algebraic geometry and number theory},
   },
   book={
      series={Progr. Math.},
      volume={253},
      publisher={Birkh\"{a}user Boston, Boston, MA},
   },
   date={2006},
   pages={407--451},
}
\bib{Ihara2}{article}{
   author={Ihara, Yasutaka},
   title={The Euler-Kronecker invariants in various families of global
   fields},
   language={English, with English and French summaries},
   conference={
      title={Arithmetics, geometry, and coding theory (AGCT 2005)},
   },
   book={
      series={S\'{e}min. Congr.},
      volume={21},
      publisher={Soc. Math. France, Paris},
   },
   date={2010},
   pages={79--102},
}

\bib{I-M}{article}{
   author={Ihara, Yasutaka},
   author={Matsumoto, Kohji},
   title={On certain mean values and the value-distribution of logarithms of
   Dirichlet $L$-functions},
   journal={Q. J. Math.},
   volume={62},
   date={2011},
   number={3},
   pages={637--677},

}

\bib{I-M1}{article}{
   author={Ihara, Yasutaka},
   author={Matsumoto, Kohji},
   title={On $\log L$ and $L'/L$ for $L$-functions and the associated
   ``$M$-functions'': connections in optimal cases},
   language={English, with English and Russian summaries},
   journal={Mosc. Math. J.},
   volume={11},
   date={2011},
   number={1},
   pages={73--111, 182},}
   
\bib{JW}{article}{
   author={Jessen, B\o rge},
   author={Wintner, Aurel},
   title={Distribution functions and the Riemann zeta function},
   journal={Trans. Amer. Math. Soc.},
   volume={38},
   date={1935},
   number={1},
   pages={48--88},
}

\bib{kowalski}{misc}{
author={Kowalski, Emmanuel},
title={Lecture Notes: An Introduction to Probabilistic Number Theory},
journal={lecture Notes},
year={ Version of March 13, 2020},
note={ \url{https://people.math.ethz.ch/~kowalski/probabilistic-number-theory.pdf}},
institution={ETH Zurich},}

\bib{lamzouri1}{article}{
   author={Lamzouri, Youness},
   title={Distribution of values of $L$-functions at the edge of the
   critical strip},
   journal={Proc. Lond. Math. Soc. (3)},
   volume={100},
   date={2010},
   number={3},
   pages={835--863},
 
}

\bib{lamzouri2}{article}{
   author={Lamzouri, Youness},
   title={On the distribution of extreme values of zeta and $L$-functions in
   the strip $\frac12<\sigma<1$},
   journal={Int. Math. Res. Not. IMRN},
   date={2011},
   number={23},
   pages={5449--5503},
 
}

\bib{lamzouri3}{article}{
   author={Lamzouri, Youness},
   title={The distribution of Euler-Kronecker constants of quadratic fields},
   journal={J. Math. Anal. Appl.},
   volume={432},
   date={2015},
   number={2},
   pages={632--653}
}

\bib{LL}{article}{
   author={Lamzouri, Youness},
   author={ Languasco, Alessandro},
   title={Small values of $\left|L'/L(1,\chi)\right|$},
   journal={Experimental Mathematics},
   date={to appear},
 
}

\bib{LLR}{article}{
   author={Lamzouri, Youness},
   author={Lester, Stephen},
   author={Radziwi\l\l, Maksym},
   title={Discrepancy bounds for the distribution of the Riemann zeta-function and applications},
   journal={Journal d'Analyse Math\'{e}matique},
   date={2019},
   volume={139},
   number={2},
   pages={453--494},
 
}

\bib{M-M}{article}{
   author={Mourtada, Mariam},
   author={Murty, V. Kumar},
   title={Distribution of values of $L'/L(\sigma,\chi_D)$},
   language={English, with English and Russian summaries},
   journal={Mosc. Math. J.},
   volume={15},
   date={2015},
   number={3},
   pages={497--509, 605},
 
}

\bib{MV}{book}{
   author={Montgomery, Hugh L.},
   author={Vaughan, Robert C.},
   title={Multiplicative Number Theory I: Classical Theory},
   series={Cambridge Studies in Advanced Mathematics},
 place={Cambridge},
 volume={97},
   publisher={Cambridge University Press},
   date={2006},
   pages={xvii+552},
  }

\bib{S}{article}{
   author={Stankus, E.},
   title={Distribution of Dirichlet $L$-functions with real characters in
   the half-plane ${\rm Re}$ $s>1/2$},
   language={Russian, with Lithuanian and English summaries},
   journal={Litovsk. Mat. Sb.},
   volume={15},
   date={1975},
   number={4},
   pages={199--214, 249},
   issn={0132-2818},
   review={\MR{0406956}},
}

\bib{Selberg}{book}{
   author={Tsang, Kai-Man},
   title={The distribution of the values of the {R}iemann zeta-function},
   note={Thesis (Ph.D) -- Princeton University},
   publisher={ProQuest LLC, Ann Arbor, MI},
   year={1984},
   pages={189},
}

\bib{T}{book}{
   author={Tenenbaum, G\'erald},
   title={Introduction to analytic and probabilistic number theory},
   series={Graduate Studies in Mathematics},
   volume={163},
   edition={3},
   note={Translated from the 2008 French edition by Patrick D. F. Ion},
   publisher={American Mathematical Society, Providence, RI},
   date={2015},
    isbn={978-0-8218-9854-3},
   pages={xxiv+629},
}



\end{rezabib}

\end{document}